\documentclass{amsart}
\usepackage{mathrsfs}
\usepackage{stmaryrd,mathtools}
\usepackage{enumerate}
\usepackage{tikz-cd}
\usepackage[all]{xy}
\usepackage{aliascnt}

\usepackage{shuffle}
\usepackage{ytableau}

\usepackage[colorlinks=true, linkcolor=blue, citecolor=blue]{hyperref}
\usepackage{fullpage}
\usepackage{verbatim}
\usepackage{amssymb}

\newtheorem{theorem}{Theorem}[section]

\newaliascnt{lemma}{theorem}
\newtheorem{lemma}[lemma]{Lemma}
\aliascntresetthe{lemma}

\newaliascnt{corollary}{theorem}
\newtheorem{corollary}[corollary]{Corollary}
\aliascntresetthe{corollary}

\newaliascnt{proposition}{theorem}
\newtheorem{proposition}[proposition]{Proposition}
\aliascntresetthe{proposition}

\newaliascnt{potato}{theorem}

\aliascntresetthe{potato}

\newaliascnt{definitionlemma}{theorem}

\aliascntresetthe{definitionlemma}

\newaliascnt{conjecture}{theorem}

\aliascntresetthe{conjecture}

\newaliascnt{claim}{theorem}
\newtheorem{claim}[claim]{Claim}
\aliascntresetthe{claim}

\newaliascnt{question}{theorem}
\newtheorem{question}[question]{Question}
\aliascntresetthe{question}

\theoremstyle{definition}

\newaliascnt{definition}{theorem}
\newtheorem{definition}[definition]{Definition}
\aliascntresetthe{definition}

\newaliascnt{remark}{theorem}
\newtheorem{remark}[remark]{Remark}
\aliascntresetthe{remark}

\newaliascnt{example}{theorem}
\newtheorem{examplex}[example]{Example}
\newenvironment{example}
  {\pushQED{\qed}\examplex}
  {\popQED\endexamplex}
\aliascntresetthe{example}

\newaliascnt{notation}{theorem}

\aliascntresetthe{notation}

\definecolor{darkblue}{rgb}{0.6,0,0.1}

\usepackage{tikz}
\usetikzlibrary{calc, shapes, backgrounds,arrows,positioning,plotmarks}
\tikzset{>=stealth',
  head/.style = {fill = white, text=black},
  plaque/.style = {draw, rectangle, minimum size = 10mm}, 
  pil/.style={->,thick},
  junct/.style = {draw,circle,inner sep=0.5pt,outer sep=0pt, fill=black}
  }

\newcommand{\PP}{\mathbb{P}}
\newcommand{\ZZ}{\mathbb{Z}}
\newcommand{\QQ}{\mathbb{Q}}
\newcommand{\NN}{\mathbb{N}}

\newcommand{\CC}{\mathbb{C}}

\newcommand{\LL}{\mathbb{L}}

\newcommand{\bL}{\mathbb{L}}

\newcommand{\bA}{\mathbb{A}}

\newcommand{\cX}{\mathcal{X}}

\newcommand{\cY}{\mathcal{Y}}

\newcommand{\cB}{\mathcal{B}}

\newcommand{\cP}{\mathcal{P}}

\newcommand{\calT}{\mathcal{T}}

\newcommand{\frakp}{\mathfrak{p}}

\DeclareMathOperator{\age}{age}
\DeclareMathOperator{\Conj}{Conj}
\DeclareMathOperator{\tors}{tors}

\DeclareMathOperator{\SL}{SL}

\DeclareMathOperator{\e}{e}

\DeclareMathOperator{\Spec}{Spec}

\makeatletter
\@namedef{subjclassname@2020}{%
  \textup{2020} Mathematics Subject Classification}
\makeatother

\newif\ifhascomments \hascommentstrue
\ifhascomments
  \newcommand{\matt}[1]{{\color{red}[[\ensuremath{\spadesuit\spadesuit\spadesuit} #1]]}}
  \newcommand{\dori}[1]{{\color{blue}[[\ensuremath{\clubsuit\clubsuit\clubsuit} #1]]}}
  \newcommand{\elden}[1]{{\color{blue}[[\ensuremath{\clubsuit\clubsuit\clubsuit} #1]]}}
   \newcommand{\evan}[1]{{\color{red}[[\ensuremath{\natural\natural\natural} #1]]}}
\else
  \newcommand{\elden}[1]{}
  \newcommand{\dori}[1]{}
  \newcommand{\matt}[1]{}
\fi

\begin{document}

\title{$K$-equivalence and integral cohomology}

\author{Matthew~Satriano}
\thanks{MS was partially supported by a Discovery Grant from the
  National Science and Engineering Research Council of Canada.}
\address[MS]{Department of Pure Mathematics, University
  of Waterloo, Waterloo ON N2L3G1, Canada}
\email{msatrian@uwaterloo.ca}

\author{Evan~Sundbo}
\thanks{}
\address[ES]{Department of Pure Mathematics, University
  of Waterloo, Waterloo ON N2L3G1, Canada}
\email{e2sundbo@uwaterloo.ca}

\date{\today}
\subjclass[2020]{}

\begin{abstract} 
We introduce an integral version of the Hodge polynomial, which encodes the integral cohomology of smooth projective varieties. We prove it extends to a function which is well-defined on the Grothendieck ring of varieties and we obtain as a consequence that $K$-equivalent smooth projective varieties have isomorphic integral cohomology groups.
\end{abstract}

\maketitle

\tableofcontents

\numberwithin{theorem}{section}
\numberwithin{lemma}{section}
\numberwithin{corollary}{section}
\numberwithin{proposition}{section}
\numberwithin{conjecture}{section}
\numberwithin{question}{section}
\numberwithin{remark}{section}
\numberwithin{definition}{section}
\numberwithin{example}{section}
\numberwithin{notation}{section}
\numberwithin{equation}{section}

\section{Introduction}


In 1995, Kontsevich \cite{Kon95} proved that $K$-equivalent smooth projective varieties have the same Hodge numbers. In this brief note, we extend his result to integral cohomology.

\begin{theorem}\label{thm:HZ-k-equivalence}
    If $X$ and $X'$ are $K$-equivalent smooth projective varieties over $\CC$ then $H^i(X,\ZZ)\simeq H^i(X',\ZZ)$ for all $i$.
\end{theorem}

We prove this theorem by introducing an \emph{integral virtual Hodge function} $H_{\text{vir},\ZZ}$. It is well-defined on the Grothendieck ring of varieties $K_0(\text{Var}_\CC)$ and has the property that when $X$ is a smooth projective variety, the integral Hodge structure, and in particular the integral cohomology groups, of $X$ can be read off from $H_{\text{vir},\ZZ}(X)$. 


The main technical point in constructing $H_{\text{vir},\ZZ}$ is defining the ring where it takes values in such a way that $H_{\text{vir},\ZZ}$ is multiplicative, i.e., that $$H_{\text{vir},\ZZ}(X\times Y) = H_{\text{vir},\ZZ}(X)H_{\text{vir},\ZZ}(Y).$$ 

\begin{definition}
Given a smooth projective complex variety $X$, let the \emph{integral Hodge function} $$H_{\ZZ}(X,s_\frakp,r_j,u,v,t,x) =\sum_{\substack{i,j\geq 0 \\ \frakp\text{ prime}}}(-1)^ia_{\frakp,i,j}(X)s_\frakp r_jt^ix  +\sum_{p,q,i\geq0}(-1)^ih^{p,q}(X)u^pv^q,$$
where 
\[
a_{\frakp,i,j}(X) = \text{rk}_{\ZZ/\frakp\ZZ}\left(\frac{\frakp^j H^i(X,\ZZ)_\text{tors}}{\frakp^{j+1} H^i(X,\ZZ)_\text{tors}}\right)
\]
and $h^{p,q}(X)$ are the usual Hodge numbers of $X$. This function takes values in the ring
 \[
R:=\frac{\ZZ[s_\frakp,r_j,u,v,t,x: \frakp\text{ prime},\,j\geq0]}{\langle tx^2=tx-x, ux=vx=tx, s_\frakp^2 = s_\frakp, r_j^2 = r_j, s_\frakp s_{\frakp'}=r_jr_{j'}=0\text{ for }\frakp\neq\frakp', j\neq j'    \rangle}.
\]
\end{definition}

We extend $H_\ZZ$ to the virtual integral Hodge function $H_{\text{vir},\ZZ}$ in Proposition \ref{checks}, which takes values on complex varieties which are not necessarily smooth or projective. We discuss the definition of the integral Hodge function and some of its basic properties in Section \ref{first-section}. Subsection \ref{section:checks} is devoted to proving that $H_\ZZ$ extends to a function $H_{\text{vir},\ZZ}$ on the Grothendieck ring of varieties. As a quick application, we give a new proof of the fact that if a smooth projective variety admits an algebraic cell decomposition (for example if $X$ admits a nontrivial $\CC^*$ action with only finitely many fixed points, per \cite{BB73}) then its integral cohomology is torsion-free. We then use the existence of $H_{\text{vir},\ZZ}$ in Section \ref{section:HZ-k-equivalence} to prove our Theorem \ref{thm:HZ-k-equivalence}.

As a corollary of Theorem \ref{thm:HZ-k-equivalence}, one finds that if $X$ is a projective algebraic variety of dimension $n$ with at worst Gorenstein canonical singularities admitting a crepant resolution of singularities $Y\to X$, then the integral cohomology groups of $Y$ are independent of the choice of crepant resolution.

Lastly, let us note that $H_{\text{vir},\ZZ}$ also admits further generalizations, such as to a \emph{stringy} version by making use of the Gorenstein measure \cite{DL02} (as is done for non-integral invariants in e.g. Section 7.3.4 of \cite{CLNS18}). One may also define a virtual integral \emph{stacky} version by extending $H_{\text{vir},\ZZ}$ to the Grothendieck group of stacks introduced by Ekedahl \cite{Eke25}. 

 After the first version of this note was posted to the arXiv we were made aware of \cite{GS96} and \cite{XY18}, in which Theorem \ref{thm:HZ-k-equivalence} was previously shown to be true.

\emph{Acknowledgements.} We thank Ted Chinburg for asking the question whether or not Theorem \ref{thm:HZ-k-equivalence} holds. We are also indebted to Peter Angelinos, Donu Arapura, Dori Bejleri, Xuemiao Chen, Michael Groechenig, Jiahui Huang, Christian Liedtke, David McKinnon, and especially Jeremy Usatine for many interesting discussions. Finally we thank Qizheng Yin for drawing our attention to the references \cite{GS96} and \cite{XY18}.

\section{Integral Virtual Hodge Functions}\label{first-section}

\begin{definition}
    Let $X$ be a smooth projective complex variety. The \emph{torsion Poincar\'e function} of $X$ is
    \[
    \calT(X) =  \sum_{\substack{i,j\geq 0 \\ \frakp\text{ prime}}}(-1)^ia_{\frakp,i,j}(X)s_\frakp r_jt^ix 
    \]
    lying in the ring
    \[
S:=\frac{\ZZ[s_\frakp,r_j,t,x: \frakp\text{ prime},\,j\geq0]}{\langle tx^2=tx-x, s_\frakp^2 = s_\frakp, r_j^2 = r_j, s_\frakp s_{\frakp'}=r_jr_{j'}=0\text{ for }\frakp\neq\frakp', j\neq j'    \rangle}
\]
    defined by
\[
a_{\frakp,i,j}(X) = \text{rk}_{\ZZ/\frakp\ZZ}\left(\frac{\frakp^j H^i(X,\ZZ)_\text{tors}}{\frakp^{j+1} H^i(X,\ZZ)_\text{tors}}\right).
\]
\end{definition}

This invariant encodes torsion in the integral cohomology groups of $X$, which can be recovered (see Proposition \ref{prop:recoverH(X,Z)}) as 
\[
H^i(X,\ZZ)_{\text{tors}}\simeq\bigoplus_{\frakp\text{ prime}}\bigoplus_{j\geq 1}\left(\ZZ/\frakp^j\ZZ\right)^{a_{\frakp,i,j-1}(X)-a_{\frakp,i,j}(X)}
\]

Note however that $\calT$ is not multiplicative, in that we do not necessarily have the equality $\calT(X\times Y) = \calT(X)\calT(Y).$ This is due to the possibility of the free parts of the cohomology groups interacting with the torsion parts. To obtain an invariant that captures torsion information and behaves in  the expected ways, we need to keep track of the free part of the cohomology groups as well.

\begin{definition} Let $X$ be a smooth projective complex variety. The \emph{integral Hodge function} of $X$ is
    \begin{equation}\label{H_Z_def}
    H_{\ZZ}(X,s_\frakp,r_j,t,x,u,v) =\calT(X) +\sum_{p,q\geq0}(-1)^{p+q}h^{p,q}(X)u^pv^q
    \end{equation}
    lying in the ring
    \[
R:=\frac{S\otimes_\ZZ\ZZ[u,v]}{\langle ux=vx=tx\rangle}
\]
where $h^{p,q}(X)$ are the usual Hodge numbers of $X$. We frequently denote $H_{\ZZ}(X,s_\frakp,r_j,t,x,u,v)$ simply as $H_{\ZZ}(X)$.
\end{definition}

We introduce the following definition; we prove it is well-defined and recovers the integral cohomology in Proposition \ref{prop:recoverH(X,Z)}.

\begin{definition}\label{def:reconstruction-map}
Denote by $R^+\subset R$
$$\bigg\{\sum_{\substack{0\leq i\leq n \\ j\geq 0,\,\frakp\text{ prime}}}(-1)^ia_{\frakp,i,j}s_\frakp r_jt^ix  +\sum_{0\leq p,q\leq n}(-1)^{p+q}h^{p,q}u^pv^q\in R \,\big|\,n\geq 0, \, h^{p,q}\geq0\text{, and }a_{\frakp,i,j-1}\geq a_{\frakp,i,j}\geq 0 \bigg\}.$$
For $\alpha\in R^+$ with notation as above, let
$$\phi_i(\alpha) = \bigoplus_{p,q\geq 0}\ZZ^{h^{p,q}}\oplus\bigoplus_{\frakp\text{ prime}}\bigoplus_{j\geq 1}\left(\ZZ/\frakp^j\ZZ\right)^{a_{\frakp,i,j-1}-a_{\frakp,i,j}}.$$
\end{definition}


\begin{proposition}\label{prop:recoverH(X,Z)}
The maps $\phi_i$ are well-defined and for $X$ smooth and projective we have
\[
\phi_i(H_{\ZZ}(X))\simeq H^i(X,\ZZ).
\]
\end{proposition}
\begin{proof}
  Write the ring $R$ as \[
R = \frac{\ZZ[s_\frakp,r_j]}{\langle s_\frakp^2=s_\frakp,r_j = r_j^2, s_\frakp s_{\frakp'}=r_jr_{j'}=0  \rangle}\otimes_\ZZ\frac{\ZZ[u,v,t,x]}{\langle tx^2=tx-x ,ux=vx=tx\rangle}.
\]
Then the first tensor factor has a basis given by $1,s_\frakp,r_j,s_\frakp r_j$. The second tensor factor has a basis given by $u^pv^qt^i$ and $t^ix$ for $p,q,i \geq 0$, as well as $x^i$ for $i\geq2$. As a result, we see the elements $s_\frakp r_jt^ix$ and $u^pv^qt^i$ form a linearly independent set in $R$ allowing us to read off the coefficients $a_{\frakp,i,j}(X)$ and $h^{p,q}$. 
\end{proof}




\section{Integral cohomology under $K$-equivalance}\label{section:K-equiv}

Our goal in this section is to prove Theorem \ref{thm:HZ-k-equivalence}. Recall that two smooth projective varieties $X$ and $Y$ are said to be \emph{$K$-equivalent} if there exists a smooth projective variety $Z$, dominating both $X$ and $Y$, birational to both $X$ and $Y$, and with $K_{Z/X} = K_{Z/Y}$.  

In this first subsection below, we prove $H_{\ZZ}$ extends to a function on the Grothendieck ring of varieties. We then use this in the following subsection to prove Theorem \ref{thm:HZ-k-equivalence}.

\subsection{$H_\ZZ$ as a function on the Grothendieck ring of varieties}\label{section:checks}

\begin{proposition}\label{checks}
    Let $X$ and $Y$ be smooth projective complex algebraic varieties. Then
    \begin{enumerate}
        \item If $Z\subset X$ is a smooth projective subvariety of codimension $c$, $\pi:\text{Bl}_Z(X)\to X$ is the blow-up of $X$ along $Z$, and $E$ is the exceptional divisor of the blowup, then $$H_{\ZZ}(X)-H_{\ZZ}(Z) = H_{\ZZ}(\text{Bl}_Z(X)) - H_{\ZZ}(E).$$
        \item $H_{\ZZ}(X\times Y) = H_{\ZZ}(X)H_{\ZZ}(Y)$ 
    \end{enumerate}
    In particular, $H_{\ZZ}$ extends to a function
    \[
    H_{\text{vir},\ZZ}\colon K_0(\text{Var}_\CC)\to R.
    \]
\end{proposition}

We first require a preliminary result.

\begin{definition}
Let $A$ be a finitely generated $\ZZ$-module which is torsion. Let
$$a_{\frakp,j}(A) = \text{rk}_{\ZZ/\frakp\ZZ}\left(\frac{\frakp^j A}{\frakp^{j+1}A }\right)$$ for each prime $\frakp$ and each $j\geq 0$. Then $A$ is reconstructed from these invariants as 
\[
A\simeq\bigoplus_{\frakp\text{ prime}}\bigoplus_{j\geq 1}\left(\ZZ/\frakp^j\ZZ\right)^{a_{\frakp,j-1}(A)-a_{\frakp,j}(A)}.
\]
\end{definition}

\begin{proposition}\label{a_facts} Let $\frakp,\frakp'$ be different primes, $r,k,k'\geq0$, and $A,B$ finitely generated torsion $\ZZ$-modules.
\begin{enumerate}
\item $a_{\frakp,j}(A\oplus B) = a_{\frakp,j}(A)+a_{\frakp,j}(B)$
\item $a_{\frakp,j}(A^{\oplus r}) = ra_{\frakp,j}(A)$
\item $a_{\frakp,j}(\ZZ/\frakp'\ZZ) =0$
\item $a_{\frakp,j}(\ZZ/\frakp^k\ZZ) = \begin{cases}
1&\text{if }j<k\\
        0&\text{otherwise}
\end{cases}$
\item $a_{\frakp,j}(\ZZ/\frakp^k\ZZ\otimes \ZZ/\frakp^{k'}\ZZ) = \begin{cases}
1&\text{if }j<\min\{k,k'\}\\
        0&\text{otherwise}
\end{cases}$
\item $a_{\frakp,j}(A\otimes B) = a_{\frakp,j}(A)a_{\frakp,j}(B)$
\end{enumerate}
\end{proposition}

\begin{proof}
Items $(1)-(5)$ are immediate from the definition. We now prove $(6)$. By part $(3)$ we only have to look at the $\frakp$-torsion parts of $A$ and $B$. Write
$$A_{\frakp-\text{tors}} = (\ZZ/\frakp\ZZ)^{r_1}\oplus\ldots\oplus(\ZZ/\frakp^k\ZZ)^{r_k}$$
and 
$$B_{\frakp-\text{tors}} = (\ZZ/\frakp\ZZ)^{r'_1}\oplus\ldots\oplus(\ZZ/\frakp^{k'}\ZZ)^{r'_{k'}}.$$
Then
\[
(A\otimes B)_{\frakp-\text{tors}} = \bigoplus_{\substack{i\leq k\\ i'\leq k'}}\left(\ZZ/\frakp^i\otimes\ZZ/\frakp^{i'}\right)^{\oplus r_ir'_{i'}}
\]
so that by $(1)$--$(5)$ we see
\[
a_{\frakp,j}(A\otimes B)=\sum_{\substack{j<i\leq k\\ j<i'\leq k'}}r_ir'_{i'}=\left(\sum_{j<i}r_i\right)\left(\sum_{j<i'}r'_{i'}\right)=a_{\frakp,j}(A)a_{\frakp,j}(B).\qedhere
\]
\end{proof}

\begin{proof}[{Proof of Proposition \ref{checks}}]

(1) By Section 7.3.3 of \cite{Voi07}, we see
\begin{equation}\label{cohom_of_blowup}
H^k(\text{Bl}_Z(X),\ZZ)=H^k(X,\ZZ)\oplus\bigoplus_{i=1}^{c-1}H^{k-2i}(Z,\ZZ).
\end{equation}
Furthermore, since $E\to Z$ is a projective bundle $P\to B$ of rank $c$, we have 
\begin{equation}\label{cohom_of_proj_bdle}
H^k(P,\ZZ) = \bigoplus_{i=0}^{c-1}H^{k-2i}(B,\ZZ).
\end{equation}
Hence (1) follows.



(2) We begin by claiming that the K\"{u}nneth formula for integral cohomology yields
\begin{align}
  H_{\ZZ}(X\times Y) &= \sum_{p,q\geq 0}\left( \sum_{\substack{ e+f=p\\ g+h=q}}(-1)^{p+q}h^{e,g}(X)h^{f,h}(Y)\right)u^pv^q+\label{firstline}\\
  &\qquad+\sum_{\substack{i,j\geq 0 \\ \frakp\text{ prime}}}(-1)^i\left( \sum_{p+q+c=i}\left(h^{p,q}(X)a_{\frakp,c,j}(Y) + a_{\frakp,c,j}(X)h^{p,q}(Y)\right)\right)s_\frakp r_jt^ix\label{secondline}\\
   &\qquad+\sum_{\substack{i,j\geq 0 \\ \frakp\text{ prime}}}(-1)^i\left(\sum_{\substack{c+d=i}}a_{\frakp,c,j}(X)a_{\frakp,d,j}(Y)+\sum_{c+d=i+1}a_{\frakp,c,j}(X)a_{\frakp,d,j}(Y)\right)s_\frakp r_jt^ix\label{thirdline}
\end{align}

The first line is the contribution from the free parts of the integral cohomology groups so let us focus on the second and third lines. For each $i\geq0$ we have

\begin{align*}
H^i(X\times Y,\ZZ)\simeq&\bigoplus_{c+d=i}\left(H^c(X,\ZZ)\otimes H^d(Y,\ZZ)\right)_{\text{tors}}\oplus
\bigoplus_{c+d=i+1}\text{Tor}_\ZZ(H^c(X,\ZZ), H^d(Y,\ZZ))\\
&\qquad=\bigoplus_{c+d=i}\left(H^c(X,\ZZ)_{\text{free}}\otimes H^d(Y,\ZZ)_{\text{tors}}\right)\oplus \bigoplus_{c+d=i}\left(H^c(X,\ZZ)_{\text{tors}}\otimes H^d(Y,\ZZ)_{\text{free}}\right)\\
&\qquad\qquad\oplus \bigoplus_{c+d=i}\left(H^c(X,\ZZ)_{\text{tors}}\otimes H^d(Y,\ZZ)_{\text{tors}}\right)\oplus \bigoplus_{c+d=i+1}\text{Tor}_\ZZ(H^c(X,\ZZ)_{\text{tors}}, H^d(Y,\ZZ)_{\text{tors}})
\end{align*}

Since $H^c(X,\ZZ)_{\text{free}}\otimes H^d(Y,\ZZ)_{\text{tors}} \simeq  H^d(Y,\ZZ)_{\text{tors}}^{\sum_{p+q=c}h^{p,q}(X)}$, we see by parts $(1)$ and $(2)$ of Proposition \ref{a_facts} that $$a_{\frakp,j}\left(\bigoplus_{c+d=i}\left(H^c(X,\ZZ)_{\text{free}}\otimes H^d(Y,\ZZ)_{\text{tors}}\right)\right)=\sum_{p+q+c=i}(-1)^ih^{p,q}(X)a_{\frakp,d,j}(Y).$$ We have a similar expression for $\bigoplus_{c+d=i}\left(H^c(X,\ZZ)_{\text{tors}}\otimes H^d(Y,\ZZ)_{\text{free}}\right)$. These two sums account for Line \eqref{secondline} above.

Since $\ZZ/n\ZZ\otimes \ZZ/m\ZZ \simeq \ZZ/\text{gcd}(n,m)\ZZ\simeq \text{Tor}_\ZZ(\ZZ/n\ZZ, \ZZ/m\ZZ)$, the contributions to the integral virtual Hodge polynomial coming from $$\bigoplus_{c+d=i}\left(H^c(X,\ZZ)_{\text{tors}}\otimes H^d(Y,\ZZ)_{\text{tors}}\right)$$ and from $$\bigoplus_{c+d=i+1}\text{Tor}_\ZZ(H^c(X,\ZZ)_{\text{tors}}, H^d(Y,\ZZ)_{\text{tors}})$$ will be the same, except that they contribute to different $t$-degree. Part $(6)$ of Proposition \ref{a_facts} then directly accounts for Line \eqref{thirdline}.

Secondly we must simplify \begin{align*}H_{\ZZ}(X)H_{\ZZ}(Y)&=\left( \sum_{\substack{i,j\geq 0 \\ \frakp\text{ prime}}}(-1)^ia_{\frakp,i,j}(X)s_\frakp r_jt^ix+\sum_{p,q\geq0}(-1)^{p+q}h^{p,q}(X)u^pv^q\right)\\
&\qquad\left(\sum_{\substack{i,j\geq 0 \\ \frakp\text{ prime}}}(-1)^ia_{\frakp,i,j}(Y)s_\frakp r_jt^ix+\sum_{p,q\geq0}(-1)^{p+q}h^{p,q}(Y)u^pv^q\right).
\end{align*}
The product of the second terms in the two factors above reads as 
   $$\sum_{p,q\geq 0}\left(\sum_{\substack{e+f=p \\ g+h=q}}(-1)^{p+q}h^{e,g}(X)h^{f,h}(Y)u^{e+f}v^{g+h}\right),$$ 
   which accounts exactly for Line \eqref{firstline}. The coefficient of $t^ix$ in the product on coming from mixing the $a$ and $h$ terms is
 \begin{align*}
     &\sum_{\substack{j,c,p,q\geq 0 \\ \frakp\text{ prime}}}\left( (-1)^{p+q+c}\left(h^{p,q}(X)a_{\frakp,c,j}(Y) + a_{\frakp,c,j}(X)h^{p,q}(Y)\right)s_\frakp r_ju^pv^qt^{c}x\right)\\
     &\qquad = \sum_{\substack{i,j\geq 0 \\ \frakp\text{ prime}}}(-1)^i\left( \sum_{p+q+c=i}\left(h^{p,q}(X)a_{\frakp,c,j}(Y) + a_{\frakp,c,j}(X)h^{p,q}(Y)\right)\right)s_\frakp r_jt^ix
 \end{align*}
which matches Line \eqref{secondline}.

 Lastly we have to inspect the part of the coefficient of $t^ix$ coming from the $a$ terms. Any terms with $\frakp\neq \frakp'$ or $j\neq j'$ vanish, and multiplying $a_{\frakp,c,j}s_\frakp r_jt^cx$ with $a_{\frakp,d,j}s_{\frakp}r_{j}t^dx$ yields 
 \begin{align*}
     &(-1)^{c+d}a_{\frakp,c,j}a_{\frakp,d,j}s_\frakp r_jt^{c+d}x^2\\
     &\qquad = (-1)^{c+d}a_{\frakp,c,j}a_{\frakp,d,j}s_\frakp r_jt^{c+d}x + (-1)^{c+d}a_{\frakp,c,j}a_{\frakp,d,j}s_\frakp r_jt^{c+d-1}x
     \end{align*}
     This accounts for the torsion terms in the Künneth formula and shows that the part of the coefficient of $t^ix$ coming from the $a$ terms is
      \begin{align*}
     &\sum_{\substack{j\geq 0 \\ \frakp\text{ prime}}}\left(\sum_{\substack{c+d=i \\e+f=i+1}}(-1)^i\left(a_{\frakp,c,j}(X)a_{\frakp,d,j}(Y)+a_{\frakp,e,j}(X)a_{\frakp,f,j}(Y)\right)s_\frakp r_jt^{i}x\right)
   \end{align*}
   and this matches Line \eqref{thirdline}.

   For the final claim, use Theorem 3.1 of \cite{Bit04} which states that the abelian group generated by the isomorphism classes of smooth projective
varieties under to the relations $[\varnothing] = 0$ and $[\text{Bl}_Z(X)]-[E] = [X]-[Z]$ is isomorphic to $K_0(\text{Var}_\CC)$.
\end{proof}

\begin{corollary}\label{cutandpaste} If $X$ is a quasi-projective variety and $U \subset X$ is an open subset then
        $$H_{\text{vir},\ZZ}(X) = H_{\text{vir},\ZZ}(U) + H_{\text{vir},\ZZ}(X\setminus U). $$
        \end{corollary}

        \begin{proof}
        As a result of $H_{\text{vir},\ZZ}$ being well-defined on $K_0(\text{Var}_\CC)$ it respects the relation $[X] = [U]-[X\setminus U]$ which is used in the standard presentation of the Grothendieck group of varieties.
        \end{proof}

        As a quick application of the existence of $H_{\text{vir},\ZZ}$ we reprove the following well-known result:

\begin{proposition}
Let $X$ be a smooth projective algebraic variety which admits an algebraic cell decomposition. Then the integral cohomology groups of $X$ are torsion-free.
\end{proposition}

\begin{proof}
By assumption we have a filtration 
\[
X=X_n\supseteq\ldots\supseteq X_0\supseteq X_{-1} = \varnothing
\]
where each $X_i$ is a closed subvariety of $X$ and each $X_i\setminus X_{i-1}$ is a disjoint union of affine spaces. Using the cut-and-paste relations (Corollary \ref{cutandpaste}) one can write $\{X\}$ as a sum of classes of affine spaces. Applying $H_{\text{vir},\ZZ}$ to the class of an affine space yields $H_{\text{vir},\ZZ}(\LL^k) = H_\ZZ(\{\PP^k\})-H_\ZZ(\{\PP^{k-1}\})= H(\{\PP^k\})-H(\{\PP^{k-1}\})$ (where $H(-)$ is the usual Hodge polynomial), implying that $H_\ZZ(X) = H(X)$ and so the integral cohomology of $X$ is torsion-free, by Proposition \ref{prop:recoverH(X,Z)}.
\end{proof}


\subsection{Proof of Theorem \ref{thm:HZ-k-equivalence}}\label{section:HZ-k-equivalence}
To prove Theorem \ref{thm:HZ-k-equivalence}, we need a notion of degree in $R$. Note that the usual $t$-degree is not well-defined in $R$, the issue being that $tx^2-tx=-x$ in $R$ does not have a well-defined degree. As a result, we must be a bit more careful. Also note that considering the original definition in Equation \eqref{H_Z_def}, an element with a coefficient of the form $u^pv^q$ should morally have degree $p+q$.

\begin{definition}
Let $\cB$ be the basis constructed in the proof of Proposition \ref{prop:recoverH(X,Z)}. Let $\alpha\in R$ be written as $\sum_{b\in\cB}n_b b$ with $n_b\in\ZZ$. The \emph{degree} of $\alpha$ is the maximum between the largest sum of powers of $t$, $u$ and $v$ appearing among the $b\in\cB$ for which $n_b\neq0$.
\end{definition}

The following result is immediate from the above definition.

\begin{lemma}
Let $\alpha,\beta\in R$ and $n\in\ZZ$. 
\begin{enumerate}
\item $\deg(\alpha+\beta)\leq \max\{\deg(\alpha),\deg(\beta)\}$
\item $\deg(n\alpha) = \deg(\alpha)$
\end{enumerate}
\end{lemma}


Denoting the class of the affine line by $\LL$ and using the cut-and-paste relations, one calculates $H_{\text{vir},\ZZ}(\LL) =uv$. Hence the function $H_{\text{vir},\ZZ}$ constructed in Proposition \ref{checks} extends naturally to a map 
\[
K_0(\text{Var}_\CC)[\LL^{-1}]\to R[(uv)^{-1}]
\]
which we continue to denote as $H_{\text{vir},\ZZ}$. Let $\widehat{K_0(\text{Var}_\CC)[\LL^{-1}]}$ be the completion with respect to the \emph{dimension filtration}
\[
\ldots\subseteq F^mK_0(\text{Var}_\CC)[\LL^{-1}]\subseteq F^{m-1}K_0(\text{Var}_\CC)[\LL^{-1}]\subseteq\ldots
\]
where $F^mK_0(\text{Var}_\CC)[\LL^{-1}]$ is the subgroup spanned by classes of the form $[X]\LL^{-i}$ with $\text{dim}X-i\geq m$. Further set 
\[
\overline{K_0(\text{Var}_\CC)[\LL^{-1}]}\subset \widehat{K_0(\text{Var}_\CC)[\LL^{-1}]}
\]
to be the image of $K_0(\text{Var}_\CC)[\LL^{-1}]$. 

\begin{lemma}\label{factors}
    The integral virtual Hodge function $H_{\text{vir},\ZZ}\colon K_0(\text{Var}_\CC)[\LL^{-1}]\to R[(uvt^2)^{-1}]$ factors through $\overline{K_0(\text{Var}_\CC)[\LL^{-1}]}\subset \widehat{K_0(\text{Var}_\CC)[\LL^{-1}]}$.
\end{lemma}

\begin{proof}
The proof is essentially the same as in the case of non-integral invariants. We want to show that if $$\alpha\in\text{ker}\left(K_0(\text{Var}_\CC)[\LL^{-1}]\to\widehat{K_0(\text{Var}_\CC)[\LL^{-1}]}\right) = \bigcap_{m\geq 0}F^mK_0(\text{Var}_\CC)[\LL^{-1}]$$
then $H_{\text{vir},\ZZ}(\alpha)=0$. By definition $F^mK_0(\text{Var}_\CC)[\LL^{-1}]$ is generated by classes of the form $\{X\}/\LL^i$ with $\text{dim}X-i\leq -m$. So, $\alpha\in F^mK_0(\text{Var}_\CC)[\LL^{-1}]$ satisfies $\deg H_{\text{vir},\ZZ}(\alpha)\leq 2\text{dim}X-2i\leq -m$ for all $m$ and thus $H_{\text{vir},\ZZ}(\alpha)=0$. 
\end{proof}

We may now turn to the main theorem of this section.

\begin{proof}[{Proof of Theorem \ref{thm:HZ-k-equivalence}}]
By \cite{Kon95} (see also, e.g., Theorem 4.3 in Chapter 6 of \cite{Pop11} for expository notes), we know that we have equality of the classes $\{X\}=\{X'\}$ in $\widehat{K_0(\text{Var}_\CC)[\LL^{-1}]}$. Thus, $\{X\}=\{X'\}$ in $\overline{K_0(\text{Var}_\CC)[\LL^{-1}]}$ and so Lemma \ref{factors} implies $H_{\text{vir},\ZZ}(X)=H_{\text{vir},\ZZ}(X')$.
\end{proof}

\begin{corollary}
Let $X$ be a projective algebraic variety of dimension $n$ with
at worst Gorenstein canonical singularities. If $X$ admits a crepant resolution of singularities $Y\to X$, the integral cohomology groups of $Y$ are independent of the choice of crepant resolution.
\end{corollary}

\bibliographystyle{alpha}
\bibliography{IntegralMcKaybib.bib}

\end{document}